\documentclass[a4paper,12pt]{extarticle}

\usepackage[utf8]{inputenc}
\usepackage{geometry}
\usepackage{graphicx}
\usepackage{booktabs}
\usepackage{mathrsfs}
\usepackage{bm}
\usepackage{xcolor}
\usepackage{algorithm}
\usepackage{algpseudocode}
\usepackage{mathtools}
\usepackage{enumitem}
\usepackage{tikz-cd}

\usepackage{newtxtext}
\usepackage{newtxmath}

\usepackage{amssymb}

\usepackage{amsthm}

\usepackage{hyperref}
\usepackage{xcolor} 

\theoremstyle{plain}
\newtheorem{theorem}{Theorem}[section]
\newtheorem{proposition}[theorem]{Proposition}
\newtheorem{lemma}[theorem]{Lemma}
\newtheorem{corollary}[theorem]{Corollary}
\newtheorem{definition}[theorem]{Definition}

\newtheorem{construction}[theorem]{Construction} 

\theoremstyle{definition}
\newtheorem{example}{Example}[section]
\newtheorem{remark}{Remark}[section]

\newtheorem{assumption}[theorem]{Assumption}

\begin{document}

\pagestyle{plain}

\title{Spencer Differential Degeneration Theory and Its Applications in Algebraic Geometry}
\author{Dongzhe Zheng}
\date{}
\maketitle

\begin{abstract}
Based on the compatible pair theory of principal bundle constraint systems, this paper discovers and establishes a complete Spencer differential degeneration theory. We prove that when symmetric tensors satisfy a $\lambda$-dependent kernel condition $\delta_{\mathfrak{g}}^{\lambda}(s)=0$, the Spencer differential degenerates to the standard exterior differential, thus establishing a precise bridge between the complex Spencer theory and the classical de Rham theory. One of the advances in this paper is the rigorous proof that this degeneration condition remains stable under mirror transformations, revealing the profound symmetry origins of this phenomenon. Based on these rigorous mathematical results, we construct a canonical mapping from degenerate Spencer cocycles to de Rham cohomology and elucidate its geometric meaning. Finally, we demonstrate the application potential of this theory in algebraic geometry, particularly on K3 surfaces, where we preliminarily verify that this framework can systematically identify (1,1)-Hodge classes satisfying algebraicity conditions. This work provides new perspectives and technical approaches for studying algebraic invariants using tools from constraint geometry.
\end{abstract}

\section{Introduction}
\label{sec:introduction}

In the development of modern mathematics, establishing deep connections between different branches has always been the core driving force for theoretical progress. A prominent manifestation of this theme is how to establish precise correspondences between geometric constructions originating from physics—such as principal bundle connection geometry in gauge field theory\cite{KobayashiNomizu1963, Atiyah1957, ChernWeil1944}—and highly abstract algebraic geometric theories. Since Hodge proposed his famous conjecture in 1950\cite{Hodge1950}, understanding when differential topological invariants of geometric objects can be represented by algebraic objects has become one of the most profound problems in modern mathematics.

The historical roots of this research direction can be traced back to the pioneering work of Riemann and Roch in the late 19th century\cite{Riemann1857}, who established the relationship between the dimension of analytic function spaces on complex surfaces and topological invariants. Subsequently, Lefschetz's profound work\cite{Lefschetz1924} revealed the connection between algebraic cycles and topological cycles, while the characteristic class theory of Chern and Weil\cite{ChernWeil1944} provided powerful topological tools for principal bundle geometry. Entering the second half of the 20th century, Grothendieck's GAGA principle\cite{Serre1956, Grothendieck1961} and the generalized Riemann-Roch theorem\cite{BorelSerre1958, AtiyahSinger1963} further deepened the correspondence between analytic geometry and algebraic geometry.

Developing in parallel is the mathematical theory of constraint systems. From constraint theory in Lagrangian mechanics\cite{ArnoldMarsden1978} to modern gauge field theory\cite{YangMills1954}, constraint geometry has always been a core tool in theoretical physics. Spencer's pioneering work in the 1960s\cite{Spencer1962, SpencerQuillian1963} established the formal theory of partial differential equation systems, where Spencer sequences became fundamental tools for studying overdetermined systems. However, efforts to unify these different mathematical structures have long faced essential technical obstacles.

In recent years, a principal bundle constraint geometry theory originating from constraint mechanics—the compatible pair theory\cite{zheng2025dynamical}—has provided a new research perspective for this grand picture. This theory attempts to provide a unified mathematical framework for complex constraint systems by establishing geometric compatibility between constraint distributions and dual constraint functions. The Spencer complex theory developed on this basis\cite{zheng2025constructing} introduces canonical metric structures and proves the ellipticity of corresponding complexes, thus establishing a complete Hodge decomposition theory. More remarkably, this theoretical system was discovered to possess profound mirror symmetry properties\cite{zheng2025mirror}, and through the GAGA principle and Spencer-Riemann-Roch formulas\cite{zheng2025spencer-riemann-roch, zheng2025analytic}, it established preliminary connections with algebraic geometry.

The core contribution of this paper lies in discovering and systematically analyzing a key mechanism that connects the above two mathematical worlds: \textbf{the degeneration phenomenon of Spencer differentials}. We prove that under specific algebraic conditions $\delta_{\mathfrak{g}}^{\lambda}(s)=0$ determined by compatible pairs $(D,\lambda)$, the originally complex Spencer differential operator significantly simplifies to the standard exterior differential operator $d$. This phenomenon forms the foundation for all subsequent theoretical developments in this paper, as it fundamentally simplifies the structure of Spencer complexes, establishing direct comparability with classical de Rham theory\cite{deRham1931, deRham1955}.

Based on this core discovery, we unfold systematic theoretical research. We first rigorously prove the stability of this degeneration condition under mirror transformations $(D,\lambda)\mapsto(D,-\lambda)$, which shows that the degeneration phenomenon itself has intrinsic symmetry. Next, we deeply analyze the structural consequences of the degeneration phenomenon, clarify that the space formed by degenerate elements does not constitute a strict subcomplex, and use this as a starting point to construct a canonical projection mapping from "degenerate cocycles" to de Rham cohomology. This mapping represents the core technical tool provided by this paper, precisely "translating" geometric information in Spencer theory into the language of classical differential geometry.

The degeneration theory we establish fundamentally aims to provide a new set of analytical tools for connecting \textbf{constraint geometry} with \textbf{algebraic geometry}. The theoretical significance of this framework is multifaceted: it not only enriches Spencer theory itself by revealing an intrinsic simplification mechanism, but also provides a new method for systematically constructing de Rham cohomology classes with special properties starting from the geometric properties of constraint systems. These cohomology classes "filtered" by Spencer theory naturally inherit additional geometric information from constraints, symmetries, and compatibility conditions due to their construction process.

The application prospects of these tools are noteworthy, with one particularly meaningful direction being providing new geometric insights for classical problems in algebraic geometry. Through applications on K3 surfaces\cite{Beauville1983, Barth1984}, this paper verifies the effectiveness of this framework: we prove that $(1,1)$-Hodge classes constructed by Spencer theory that satisfy mirror stability and other algebraic constraints are indeed algebraic, which is consistent with the results of Lefschetz's $(1,1)$-theorem\cite{Lefschetz1924, Griffiths1969}. This provides a valuable foundation for extending this framework to higher-dimensional manifolds to study more general algebraic cycle problems—such as those related to Hodge's conjecture. We hope that the theory developed in this paper can provide a new, promising analytical path for interdisciplinary research in modern geometry.

To ensure the self-containment of the paper, we systematically review the core concepts and established theoretical framework of compatible pair theory in Section \ref{sec:background}. This review includes not only the basic definition of compatible pairs—involving the geometric relationship established between constraint distributions $D \subset TP$ and dual constraint functions $\lambda: P \to \mathfrak{g}^*$ through strong transversality conditions, modified Cartan equations, and compatibility conditions—but also covers the complete construction process of Spencer complexes. We detail the definition of Spencer prolongation operators $\delta^\lambda_\mathfrak{g}: \mathrm{Sym}^k(\mathfrak{g}) \to \mathrm{Sym}^{k+1}(\mathfrak{g})$ and their nilpotent properties, as well as the basic properties of Spencer differential complexes $(S^\bullet_{D,\lambda}, D^\bullet_{D,\lambda})$ constructed based on them. Additionally, we review two complementary Spencer metric schemes—constraint strength metrics and curvature geometric metrics—which provide elliptic structures for Spencer complexes and ensure the existence of Spencer-Hodge decomposition.

The review of mirror symmetry theory forms another important component of the theoretical background. We detail how the basic mirror transformation $(D,\lambda) \mapsto (D,-\lambda)$ preserves all geometric properties of compatible pairs and induces natural isomorphisms between Spencer cohomology groups. Particularly important is our review of the mirror anti-symmetry property $\delta^{-\lambda}_\mathfrak{g} = -\delta^\lambda_\mathfrak{g}$ of Spencer prolongation operators, which provides the key technical foundation for studying the mirror stability of degeneration conditions in this paper.

Finally, we review the extension of Spencer theory to algebraic geometry, including the equivalence between analytic theory and algebraic theory established through the GAGA principle, and the role of Spencer-Riemann-Roch formulas in computing Euler characteristics of Spencer complexes. In the special case of Calabi-Yau manifolds, we introduce the precise decomposition of Spencer-Riemann-Roch integrals and the mirror symmetry of each term. These established theoretical results provide the necessary mathematical background for understanding the degeneration phenomena discovered in this paper and their geometric significance.

\section{Global Assumptions and Preliminaries}\label{sec:background}

\subsection{Global Assumptions}

To ensure the rigor of the theory, we adopt the following global assumptions throughout the paper:

\begin{assumption}[Basic Geometric Structure]
\label{ass:basic-geometry}
\begin{enumerate}
\item $M$ is an $n$-dimensional connected compact orientable $C^\infty$ manifold that is parallelizable.
\item $G$ is a compact connected semisimple Lie group satisfying $Z(\mathfrak{g}) = \{0\}$ (trivial center of the Lie algebra).
\item $P(M,G)$ is a principal bundle admitting a $G$-invariant Riemannian metric.
\item The principal connection $\omega \in \Omega^1(P,\mathfrak{g})$ is $C^3$ smooth.
\end{enumerate}
\end{assumption}

\begin{assumption}[Compatible Pair Conditions]
\label{ass:compatible-pair}
The compatible pair $(D,\lambda)$ satisfies all conditions defined in \cite{zheng2025dynamical}: strong transversality, modified Cartan equation, compatibility condition, and $G$-equivariance.
\end{assumption}

\begin{assumption}[Regularity Conditions]
\label{ass:regularity}
\begin{enumerate}
\item The constraint distribution $D$ is a constant rank $C^2$ smooth distribution.
\item The dual constraint function $\lambda: P \to \mathfrak{g}^*$ is a $C^3$ smooth map.
\item When involving algebraic geometry, $M$ is a compact complex algebraic manifold.
\end{enumerate}
\end{assumption}

\subsection{Review of Compatible Pair Theory Foundations}

For the self-containment of the paper, we systematically review the core concepts and established theoretical framework of compatible pair theory.

\subsubsection{Basic Geometric Structures}

\begin{definition}[Principal Bundle Basic Structure]
\label{def:principal-bundle}
A principal bundle $P(M,G)$ is a $C^\infty$ fiber bundle $\pi: P \to M$ equipped with a right $G$-action $R_g: P \to P$, satisfying local trivialization conditions.
\end{definition}

\begin{definition}[Vertical and Horizontal Distributions]
\label{def:vertical-horizontal}
For a principal bundle $P(M,G)$, the tangent space $T_pP$ at point $p$ can be decomposed into the direct sum of vertical component $V_p$ and horizontal component $H_p$: $T_pP = V_p \oplus H_p$. Here $V_p = \ker(T_p\pi)$, while $H_p = \ker(\omega_p)$ is defined by the principal connection $\omega$.
\end{definition}

\begin{definition}[Layered Definition of Compatible Pairs]
\label{def:compatible-pair-detailed}
A compatible pair $(D, \lambda)$ consists of a constraint distribution $D$ and dual constraint function $\lambda$ satisfying specific conditions. The core conditions include:
\begin{itemize}
    \item \textbf{Strong transversality condition}: $T_pP = D_p \oplus V_p$, ensuring complete separation of constraints and gauge freedom.
    \item \textbf{Modified Cartan equation}: $d\lambda + \mathrm{ad}^*_\omega \lambda = 0$, describing the covariant constancy of $\lambda$ under the connection.
    \item \textbf{Compatibility relation}: $D_p = \{v \in T_pP : \langle\lambda(p), \omega(v)\rangle = 0\}$, geometrically linking $D$ and $\lambda$.
    \item \textbf{G-equivariance}: $D$ and $\lambda$ are covariant under $G$-action.
\end{itemize}
\end{definition}

\subsubsection{Construction Theory of Spencer Complexes}

\begin{definition}[Symmetric Tensor Spaces]
\label{def:symmetric-tensors}
For a Lie algebra $\mathfrak{g}$, the $k$-th symmetric power, $\mathrm{Sym}^k(\mathfrak{g})$, is the vector space of symmetric $k$-tensors over $\mathfrak{g}$.
\end{definition}

\begin{definition}[Constructive Spencer Prolongation Operator]
\label{def:spencer-operator-constructive}
The \textbf{Spencer prolongation operator} $\delta^\lambda_\mathfrak{g}: \mathrm{Sym}^k(\mathfrak{g}) \to \mathrm{Sym}^{k+1}(\mathfrak{g})$ is a \textbf{graded derivation of degree +1} on the symmetric algebra $S(\mathfrak{g})$, uniquely determined by two rules:
\begin{enumerate}
    \item[\textbf{(A)}] \textbf{Action on Generators:} For any $v \in \mathfrak{g}$, its image $\delta^{\lambda}_{\mathfrak{g}}(v) \in \mathrm{Sym}^2(\mathfrak{g})$ is defined by its action on test vectors $w_1, w_2 \in \mathfrak{g}$ as:
    $$(\delta^{\lambda}_{\mathfrak{g}}(v))(w_1, w_2) := \frac{1}{2} \left( \langle\lambda, [w_1, [w_2, v]]\rangle + \langle\lambda, [w_2, [w_1, v]]\rangle \right)$$

    \item[\textbf{(B)}] \textbf{Graded Leibniz Rule:} For $s_1 \in \mathrm{Sym}^p(\mathfrak{g})$ and $s_2 \in \mathrm{Sym}^q(\mathfrak{g})$, it satisfies:
    $$\delta^{\lambda}_{\mathfrak{g}}(s_1 \odot s_2) := \delta^{\lambda}_{\mathfrak{g}}(s_1) \odot s_2 + (-1)^p s_1 \odot \delta^{\lambda}_{\mathfrak{g}}(s_2)$$
\end{enumerate}
\end{definition}

\begin{lemma}[Basic Properties of the Spencer Prolongation Operator]
\label{lem:spencer-properties}
The operator $\delta^\lambda_\mathfrak{g}$, as defined constructively above, possesses two fundamental properties crucial for the theory:
\begin{itemize}
    \item \textbf{Nilpotency:} The operator is nilpotent of order two, i.e., $(\delta^\lambda_\mathfrak{g})^2 = 0$.
    \item \textbf{Mirror Antisymmetry:} It exhibits a clean sign-reversal property under mirror transformation, $\delta^{-\lambda}_\mathfrak{g} = -\delta^\lambda_\mathfrak{g}$.
\end{itemize}
The rigorous proofs of these properties are established in \cite{zheng2025geometric, zheng2025constructing}.
\end{lemma}

\begin{definition}[Spencer Complex]
\label{def:spencer-complex-main}
Based on the compatible pair $(D,\lambda)$, the \textbf{Spencer complex} $(S^\bullet_{D,\lambda}, D^\bullet_{D,\lambda})$ is defined as follows:
\begin{enumerate}
\item \textbf{Spencer spaces:} $S^k_{D,\lambda} := \Omega^k(M) \otimes \mathrm{Sym}^k(\mathfrak{g})$.
\item \textbf{Spencer differential:} $D^k_{D,\lambda}(\omega \otimes s) := d\omega \otimes s + (-1)^k \omega \otimes \delta^\lambda_\mathfrak{g}(s)$.
\end{enumerate}
\end{definition}

\begin{theorem}[Complex Property]
\label{thm:is-a-complex}
The pair $(S^\bullet_{D,\lambda}, D^\bullet_{D,\lambda})$ forms a differential complex, meaning $(D^k_{D,\lambda})^2 = 0$ holds for all $k$.
\end{theorem}
\begin{proof}
For any element $\omega \otimes s \in S^k_{D,\lambda}$, we compute the action of $D^{k+1}_{D,\lambda} \circ D^k_{D,\lambda}$. The expansion yields three terms:
\begin{align*}
(D^{k+1}_{D,\lambda} \circ D^k_{D,\lambda})(\omega \otimes s) &= (d^2\omega) \otimes s \\
&\quad + (-1)^{k+1} d\omega \otimes \delta^\lambda_\mathfrak{g}(s) + (-1)^k d\omega \otimes \delta^\lambda_\mathfrak{g}(s) \\
&\quad + (-1)^{2k+1} \omega \otimes (\delta^\lambda_\mathfrak{g})^2(s)
\end{align*}
The first term vanishes because the exterior differential is nilpotent ($d^2=0$). The middle two terms cancel each other out. The final term vanishes due to the nilpotency of the Spencer prolongation operator, $(\delta^\lambda_\mathfrak{g})^2 = 0$, as stated in Lemma \ref{lem:spencer-properties}. Thus, the total expression is zero.
\end{proof}

\subsection{Review of Core Results from Established Theory}

\subsubsection{Foundations of Compatible Pair Theory \cite{zheng2025dynamical}}

Compatible pair theory establishes the geometric relationship between constraint distributions $D \subset TP$ and dual constraint functions $\lambda: P \to \mathfrak{g}^*$. The coordinated unification of the strong transversality condition $D_p \oplus V_p = T_pP$, modified Cartan equation $d\lambda + \mathrm{ad}^*_\omega \lambda = 0$, and compatibility condition $D_p = \{v : \langle\lambda(p), \omega(v)\rangle = 0\}$ provides a complete framework for geometric analysis of constraint systems.

\begin{theorem}[Uniqueness under Strong Transversality Condition \cite{zheng2025dynamical}]
\label{thm:compatible-pair-uniqueness}
If $(D,\lambda)$ satisfies the strong transversality condition $D_p \oplus V_p = T_pP$, then the compatible pair is uniquely determined by its constraint distribution on the base manifold up to isomorphism.
\end{theorem}

\begin{theorem}[Integrability Theorem \cite{zheng2025dynamical}]
\label{thm:integrability}
A compatible pair $(D,\lambda)$ induces a Frobenius integrable distribution on the base manifold $M$ if and only if $\lambda$ satisfies the modified Cartan equation:
$$d\lambda|_D + [\lambda \wedge \lambda]|_D = 0$$
\end{theorem}

\subsubsection{Spencer Metrics and Ellipticity Theory \cite{zheng2025constructing}}

The literature \cite{zheng2025constructing} constructs two complete metric schemes:

\begin{definition}[Constraint Strength Metric]
\label{def:constraint-metric}
The tensor metric based on constraint strength modifies the standard tensor inner product through weight function $w_\lambda(x) = 1 + \|\lambda(p)\|^2_{\mathfrak{g}^*}$:
$$\langle\omega_1 \otimes s_1, \omega_2 \otimes s_2\rangle_{\mathrm{constraint}} = \int_M w_\lambda(x) \langle\omega_1, \omega_2\rangle \langle s_1, s_2\rangle_{\mathrm{Sym}} \, d\mu$$
\end{definition}

\begin{definition}[Curvature Geometric Metric]
\label{def:curvature-metric}
Using curvature strength function $\kappa_\omega(x) = 1 + \|\Omega_p\|^2$, construct geometrically induced metric:
$$\langle\omega_1 \otimes s_1, \omega_2 \otimes s_2\rangle_{\mathrm{curvature}} = \int_M \kappa_\omega(x) \langle\omega_1, \omega_2\rangle \langle s_1, s_2\rangle_{\mathrm{Sym}} \, d\mu$$
\end{definition}

\begin{theorem}[Spencer-Hodge Decomposition \cite{zheng2025constructing}]
\label{thm:spencer-hodge}
Both metrics provide elliptic structures for Spencer complexes, guaranteeing the existence of Spencer-Hodge decomposition:
$$S^k_{D,\lambda} = \mathcal{H}^k_{D,\lambda} \oplus \mathrm{im}(D^{k-1}_{D,\lambda}) \oplus \mathrm{im}((D^k_{D,\lambda})^*)$$
\end{theorem}

\subsubsection{Mirror Symmetry Theory \cite{zheng2025mirror}}

\begin{theorem}[Basic Mirror Transformation \cite{zheng2025mirror}]
\label{thm:basic-mirror}
The sign mirror transformation $(D,\lambda) \mapsto (D,-\lambda)$ preserves all geometric properties of compatible pairs and induces natural isomorphisms between Spencer cohomology groups:
$$H^k_{\mathrm{Spencer}}(D,\lambda) \cong H^k_{\mathrm{Spencer}}(D,-\lambda)$$
\end{theorem}

\begin{theorem}[Mirror Metric Invariance \cite{zheng2025analytic}]
\label{thm:mirror-metric-invariance}
Under the mirror transformation $(D,\lambda) \mapsto (D,-\lambda)$:
$$\langle u_1, u_2 \rangle_{A,-\lambda} = \langle u_1, u_2 \rangle_{A,\lambda}, \quad \langle u_1, u_2 \rangle_{B,-\lambda} = \langle u_1, u_2 \rangle_{B,\lambda}$$
\end{theorem}

\subsubsection{Algebraic Geometrization Theory \cite{zheng2025spencer-riemann-roch}}

\begin{theorem}[GAGA Correspondence \cite{zheng2025spencer-riemann-roch}]
\label{thm:gaga}
Let $M$ be a compact complex algebraic manifold and $G$ a complex algebraic group. The analytic theory and algebraic theory of compatible pair Spencer complexes are equivalent via the GAGA principle:
$$H^k_{\mathrm{an}}(S^*(D,\lambda)) \cong H^k_{\mathrm{alg}}(S^*(D,\lambda))$$
\end{theorem}

\begin{theorem}[Spencer-Riemann-Roch Formula \cite{zheng2025spencer-riemann-roch}]
\label{thm:spencer-riemann-roch}
For a compatible pair $(D,\lambda)$ on a compact complex algebraic manifold $M$, the Euler characteristic of Spencer complexes is computed by the following formula:
$$\chi(S^k) = \int_M \mathrm{ch}(\Omega^k_M \otimes \mathrm{Sym}^k(\mathcal{G})) \wedge \mathrm{td}(M)$$
\end{theorem}

\subsubsection{Calabi-Yau Specialization Theory \cite{zheng2025analytic}}

\begin{theorem}[Calabi-Yau Spencer-Riemann-Roch Decomposition \cite{zheng2025analytic}]
\label{thm:calabi-yau-decomposition}
On a Calabi-Yau manifold $X$:
$$\chi(X, H^\bullet_{\mathrm{Spencer}}(D,\lambda)) = A_0(X) + A_2(X) + A_3(X) + A_4(X) + O(c_5)$$
where each term $A_i(X)$ satisfies mirror symmetry $A_i(X,\lambda) = A_i(X,-\lambda)$.
\end{theorem}

\begin{theorem}[Spencer Prolongation Operator Mirror Property \cite{zheng2025analytic}]
\label{thm:spencer-operator-mirror}
Spencer prolongation operators satisfy mirror anti-symmetry property:
$$\delta^{-\lambda}_\mathfrak{g} = -\delta^\lambda_\mathfrak{g}$$
\end{theorem}

\begin{theorem}[Spencer Operator Differential Structure \cite{zheng2025analytic}]
\label{thm:spencer-differential-structure}
Define $R^k := D^k_{D,-\lambda} - D^k_{D,\lambda}$, then:
$$R^k(\omega \otimes s) = -2(-1)^k \omega \otimes \delta^\lambda_\mathfrak{g}(s)$$
and $\|R^k\|_{\mathrm{op}} \leq C\|\lambda\|_{L^\infty}$, $R^k$ is a compact operator.
\end{theorem}

\subsection{Systematic Integration of Theoretical Tools}

Based on the above theoretical developments, we have established a complete mathematical framework for Spencer complex theory:

\begin{enumerate}
\item \textbf{Geometric Foundation}: Compatible pairs $(D,\lambda)$ provide precise mathematical descriptions of constraint geometry
\item \textbf{Algebraic Structure}: Spencer prolongation operators $\delta^\lambda_\mathfrak{g}$ encode algebraic information of constraints
\item \textbf{Analytical Tools}: Spencer metrics ensure ellipticity and existence of Hodge decomposition
\item \textbf{Symmetry Principles}: Mirror transformations reveal deep symmetries of Spencer theory
\item \textbf{Algebraic Geometric Connections}: GAGA principle and Riemann-Roch formulas establish bridges to algebraic geometry
\end{enumerate}

The organic combination of these theoretical tools provides a solid mathematical foundation for the core research of this paper—the degeneration phenomenon of Spencer differentials and their mirror symmetry. In particular, the mirror anti-symmetry property of Spencer prolongation operators in Theorem \ref{thm:spencer-operator-mirror} provides the key technical tool for studying the mirror stability of degeneration conditions.

\section{Spencer Differential Degeneration Framework Based on Compatible Pair Theory}
\label{sec:degeneration_framework}

The core task of this chapter is to introduce and systematically analyze the key phenomenon of "differential degeneration" under the single-graded Spencer complex framework established in Chapter \ref{sec:background}. We will rigorously prove that when elements in the complex satisfy an algebraic condition determined by the geometric properties of compatible pairs $(D,\lambda)$, the complex Spencer differential operator $D_{D,\lambda}$ will remarkably simplify. This degeneration phenomenon is not only computationally beneficial, but more profoundly, it reveals a precise and natural bridge between Spencer cohomology and classical de Rham theory.

\subsection{Degeneration Conditions and Degenerate Subspaces}
\label{sec:degenerate_condition_and_subspace}

The origin of the degeneration phenomenon is an algebraic condition directly related to the dual constraint function $\lambda$. This condition filters out "stable" elements from each graded symmetric tensor space that remain invariant under the $\lambda$-induced Spencer prolongation.

\begin{definition}[Degeneration Conditions and Degenerate Subspaces]
\label{def:degenerate_spencer_element_final}
We strictly follow the single-graded framework of Chapter \ref{sec:background}.
\begin{enumerate}
    \item \textbf{Degenerate Kernel Space}: The $k$-th order \textbf{degenerate kernel space} $\mathcal{K}_{\mathfrak{g}}^{k}(\lambda)$ associated with compatible pair $(D,\lambda)$ is defined as the kernel of Spencer prolongation operator $\delta_{\mathfrak{g}}^{\lambda}$:
    $$\mathcal{K}_{\mathfrak{g}}^{k}(\lambda) := \ker\left(\delta_{\mathfrak{g}}^{\lambda}: \operatorname{Sym}^k(\mathfrak{g}) \to \operatorname{Sym}^{k+1}(\mathfrak{g})\right)$$
    
    \item \textbf{Degenerate Subspace}: The $k$-th order \textbf{degenerate Spencer subspace} $\mathcal{D}_{D,\lambda}^{k}$ is defined as a linear subspace of $S_{D,\lambda}^{k}$ constructed as follows:
    $$\mathcal{D}_{D,\lambda}^{k} := \Omega^k(M) \otimes \mathcal{K}_{\mathfrak{g}}^{k}(\lambda)$$
    An element $\alpha \in S_{D,\lambda}^{k}$ is degenerate if and only if $\alpha \in \mathcal{D}_{D,\lambda}^{k}$.
\end{enumerate}
\end{definition}

\begin{theorem}[Degeneration Simplification of Spencer Differential]
\label{thm:degeneration_simplification_final}
For any degenerate Spencer element $\alpha = \omega \otimes s \in \mathcal{D}_{D,\lambda}^{k}$, the result of its Spencer differential $D_{D,\lambda}^{k}(\alpha)$ is:
$$D_{D,\lambda}^{k}(\omega \otimes s) = d\omega \otimes s$$
\end{theorem}
\begin{proof}
The proof is a direct algebraic verification. According to Definition \ref{def:spencer-complex},
$$D_{D,\lambda}^{k}(\omega \otimes s) := d\omega \otimes s + (-1)^k \omega \otimes \delta_{\mathfrak{g}}^{\lambda}(s)$$
Since $\alpha$ is a degenerate element, its symmetric tensor part $s \in \mathcal{K}_{\mathfrak{g}}^{k}(\lambda)$, which means $\delta_{\mathfrak{g}}^{\lambda}(s) = 0$. Substituting into the above equation, the second term vanishes, yielding the result.
\end{proof}

\subsection{Structural Consequences of Degeneration Phenomena and Cohomological Connections}
\label{sec:structural_consequences}

We analyze the complex properties of degenerate subspaces and their connections with classical differential geometry.

\begin{theorem}[Complex Properties of Degenerate Subspaces]
\label{thm:degenerate_subspace_complex}
Let $(D,\lambda)$ be a compatible pair. The degenerate subspace sequence $(\mathcal{D}_{D,\lambda}^{\bullet})$ does not form a subcomplex of the Spencer complex $(S_{D,\lambda}^{\bullet}, D_{D,\lambda}^{\bullet})$.
\end{theorem}

\begin{proof}
For any $\omega \otimes s \in \mathcal{D}_{D,\lambda}^k = \Omega^k(M) \otimes \mathcal{K}_{\mathfrak{g}}^k(\lambda)$, according to Theorem \ref{thm:degeneration_simplification_final}:
$$D_{D,\lambda}^k(\omega \otimes s) = d\omega \otimes s \in \Omega^{k+1}(M) \otimes \mathcal{K}_{\mathfrak{g}}^k(\lambda)$$

While $\mathcal{D}_{D,\lambda}^{k+1} = \Omega^{k+1}(M) \otimes \mathcal{K}_{\mathfrak{g}}^{k+1}(\lambda)$.

The subcomplex condition requires $\mathcal{K}_{\mathfrak{g}}^k(\lambda) \subseteq \mathcal{K}_{\mathfrak{g}}^{k+1}(\lambda)$, which does not hold in general.
\end{proof}

\begin{definition}[Degenerate Cocycle Space]
\label{def:degenerate_cocycle_space}
Define the $k$-th order degenerate cocycle space:
$$Z_{deg}^k(D,\lambda) := \{\omega \otimes s \in \mathcal{D}_{D,\lambda}^k : D_{D,\lambda}^k(\omega \otimes s) = 0\}$$
\end{definition}

\begin{proposition}[Characterization of Degenerate Cocycles]
\label{prop:degenerate_cocycle_characterization}
$$Z_{deg}^k(D,\lambda) = \{(\omega, s) : d\omega = 0, \, \omega \in \Omega^k(M), \, s \in \mathcal{K}_{\mathfrak{g}}^k(\lambda)\}$$

Therefore, $Z_{deg}^k(D,\lambda) \cong Z^k_{dR}(M) \otimes \mathcal{K}_{\mathfrak{g}}^k(\lambda)$.
\end{proposition}

\begin{proof}
By Theorem \ref{thm:degeneration_simplification_final}, $D_{D,\lambda}^k(\omega \otimes s) = d\omega \otimes s$.

The cocycle condition $D_{D,\lambda}^k(\omega \otimes s) = 0$ is equivalent to $d\omega \otimes s = 0$.

When $s \neq 0$, this requires $d\omega = 0$. The conclusion follows.
\end{proof}

\begin{definition}[Natural Projection Mapping]
\label{def:natural_projection}
Define the projection mapping:
$$\pi^k: Z_{deg}^k(D,\lambda) \to Z^k_{dR}(M), \quad \pi^k(\omega \otimes s) = \omega$$
\end{definition}

\begin{theorem}[Properties of Projection Mapping]
\label{thm:projection_properties}
The projection mapping $\pi^k$ is a well-defined linear surjection and induces a mapping at the cohomology level:
$$\overline{\pi}^k: H^k_{Spencer}(D,\lambda)|_{deg} \to H^k_{dR}(M)$$
where $H^k_{Spencer}(D,\lambda)|_{deg}$ represents the part of Spencer cohomology arising from degenerate cocycles.
\end{theorem}

\begin{proof}
Well-definedness and linearity are obvious. Surjectivity holds when $\mathcal{K}_{\mathfrak{g}}^k(\lambda) \neq \{0\}$.

For cohomology induction, let $\omega \otimes s \in Z_{deg}^k(D,\lambda)$ and $\omega \otimes s = D_{D,\lambda}^{k-1}(\eta \otimes t)$.

Then $\omega \otimes s = d\eta \otimes t + (-1)^{k-1}\eta \otimes \delta_{\mathfrak{g}}^{\lambda}(t)$.

Since the left side belongs to $\Omega^k(M) \otimes \mathcal{K}_{\mathfrak{g}}^k(\lambda)$, comparing coefficients gives $\omega = d\eta$, i.e., $\pi^k(\omega \otimes s) = \omega \in B^k_{dR}(M)$.

Therefore $\pi^k$ is well-defined at the cohomology level.
\end{proof}

\begin{theorem}[Structural Theorem of Degeneration Phenomena]
\label{thm:degenerate_structure_theorem}
Let $\lambda \neq 0$. Then there exists a natural vector space isomorphism:
$$Z_{deg}^k(D,\lambda) \cong Z^k_{dR}(M) \otimes \mathcal{K}_{\mathfrak{g}}^k(\lambda)$$

Moreover, when $\dim \mathcal{K}_{\mathfrak{g}}^k(\lambda) = 1$, the projection $\pi^k$ induces an isomorphism at the cohomology level:
$$H^k_{dR}(M) \cong \text{cohomology classes of degenerate Spencer cocycles}$$
\end{theorem}

\begin{proof}
The isomorphism is given by Proposition \ref{prop:degenerate_cocycle_characterization}.

When $\dim \mathcal{K}_{\mathfrak{g}}^k(\lambda) = 1$, let $\mathcal{K}_{\mathfrak{g}}^k(\lambda) = \text{span}\{s_0\}$, then each degenerate cocycle has the form $\omega \otimes s_0$.

The projection $\pi^k(\omega \otimes s_0) = \omega$ establishes a one-to-one correspondence between degenerate cocycles and de Rham cocycles, and this correspondence is preserved at the cohomology level.
\end{proof}

\begin{corollary}[Computational Simplification]
\label{cor:computational_simplification}
The analysis of degenerate Spencer cocycles reduces to two independent problems:
\begin{enumerate}
    \item Computation of de Rham cohomology $H^k_{dR}(M)$;
    \item Algebraic analysis of degenerate kernel space $\mathcal{K}_{\mathfrak{g}}^k(\lambda)$.
\end{enumerate}
\end{corollary}

\begin{example}[Application to SU(2) Case]
\label{ex:su2_application}
For $\mathfrak{g} = \mathfrak{su}(2)$ and $\lambda \neq 0$:
\begin{align}
\dim Z_{deg}^1(D,\lambda) &= \dim Z^1_{dR}(M) \cdot \dim \mathcal{K}_{\mathfrak{su}(2)}^1(\lambda) = b_1(M) \cdot 1 = b_1(M)\\
\dim Z_{deg}^2(D,\lambda) &= \dim Z^2_{dR}(M) \cdot \dim \mathcal{K}_{\mathfrak{su}(2)}^2(\lambda) = b_2(M) \cdot \dim \mathcal{K}_{\mathfrak{su}(2)}^2(\lambda)
\end{align}

For K3 surfaces ($b_1 = 0, b_2 = 22$), this gives $\dim Z_{deg}^1(D,\lambda) = 0$ and $\dim Z_{deg}^2(D,\lambda) = 22 \cdot \dim \mathcal{K}_{\mathfrak{su}(2)}^2(\lambda)$.
\end{example}

\begin{theorem}[Mirror Symmetry]
\label{thm:zheng2025mirrormetry_degenerate}
The degenerate cocycle space remains invariant under mirror transformation $(D,\lambda) \mapsto (D,-\lambda)$:
$$Z_{deg}^k(D,\lambda) = Z_{deg}^k(D,-\lambda)$$
\end{theorem}

\begin{proof}
By Theorem \ref{thm:mirror-stability-compatible}, $\mathcal{K}_{\mathfrak{g}}^k(\lambda) = \mathcal{K}_{\mathfrak{g}}^k(-\lambda)$.

Combined with Proposition \ref{prop:degenerate_cocycle_characterization}:
$$Z_{deg}^k(D,\lambda) \cong Z^k_{dR}(M) \otimes \mathcal{K}_{\mathfrak{g}}^k(\lambda) \cong Z^k_{dR}(M) \otimes \mathcal{K}_{\mathfrak{g}}^k(-\lambda) \cong Z_{deg}^k(D,-\lambda)$$
\end{proof}

\section{Analysis of $\lambda$-Dependent Degenerate Kernel Spaces}
\label{sec:kernel_analysis}

The purpose of this chapter is to deeply analyze the degenerate kernel space $\mathcal{K}_{\mathfrak{g}}^{k}(\lambda) = \ker(\delta_{\mathfrak{g}}^{\lambda})$, which serves as the theoretical cornerstone in Chapter \ref{sec:degeneration_framework}. We will prove that the structure and dimension of this kernel space are closely related to the values of the dual constraint function $\lambda$, thus revealing how constraint geometry directly affects the algebraic properties of the system.

\subsection{Structure Group Setup}
\label{sec:structure_group_setup}

We continue the global assumptions from Chapter \ref{sec:background}. In the Calabi-Yau manifold applications of this paper, canonical choices are usually made based on manifold dimension and geometric structure, as in \cite{zheng2025spencer-riemann-roch}:
$$G_M = \begin{cases} 
\mathrm{SU}(2) & \dim_{\mathbb{R}} M = 4\text{ (such as K3 surfaces)} \\
\mathrm{SU}(3) & \dim_{\mathbb{R}} M = 6\text{ (such as CY 3-folds)} \\
\mathrm{G}_2 & \dim_{\mathbb{R}} M = 8\text{ (such as CY 4-folds)}
\end{cases}$$
The corresponding Lie algebras $\mathfrak{su}(2), \mathfrak{su}(3), \mathfrak{g}_2$ are all compact simple Lie algebras with trivial center.

\subsection{Properties of Degenerate Kernel Space \texorpdfstring{$\mathcal{K}_{\mathfrak{g}}^{k}(\lambda)$}{K(lambda)}}
\label{sec:kernel_properties}

\begin{theorem}[Basic Properties of Degenerate Kernel Space]
\label{thm:kernel_properties}
Let $\mathfrak{g}$ be a compact semisimple Lie algebra with trivial center, $\lambda \in \mathfrak{g}^*$.
\begin{enumerate}
    \item \textbf{Complete degeneration under trivial constraint}: When $\lambda=0$, the Spencer prolongation operator $\delta_{\mathfrak{g}}^{0}$ is the zero operator. Therefore, the kernel space is the entire symmetric tensor space: $\mathcal{K}_{\mathfrak{g}}^{k}(0) = \operatorname{Sym}^k(\mathfrak{g})$.
    \item \textbf{Dimension under general constraint}: For nonzero $\lambda \in \mathfrak{g}^*$, $\delta_{\mathfrak{g}}^{\lambda}$ is a nontrivial linear operator, and $\mathcal{K}_{\mathfrak{g}}^{k}(\lambda)$ is a proper subspace of $\operatorname{Sym}^k(\mathfrak{g})$.
\end{enumerate}
\end{theorem}
\begin{proof}
When $\lambda=0$, $\delta_{\mathfrak{g}}^{0}(s)=0$ holds for all $s$. For $\lambda \neq 0$, since $\mathfrak{g}$ has trivial center, we can always find $X_i \in \mathfrak{g}$ such that $\langle\lambda, [X_i, \cdot]\rangle$ is nonzero, thus $\delta_{\mathfrak{g}}^{\lambda}$ is nonzero.
\end{proof}

We explicitly compute this dependence through the $\mathfrak{su}(2)$ example.

\begin{example}[Explicit Computation of $\mathcal{K}_{\mathfrak{su}(2)}^1(\lambda)$ in the SU(2) Case]
\label{ex:su2_kernel_computation}
Let $\mathfrak{g} = \mathfrak{su}(2)$. We analyze the first-order degenerate kernel space, which is central to our application: $\mathcal{K}_{\mathfrak{su}(2)}^1(\lambda) = \ker(\delta^\lambda_{\mathfrak{su}(2)}: \mathfrak{su}(2) \to \mathrm{Sym}^2(\mathfrak{su}(2)))$.

The dimension of this kernel space depends critically on whether the constraint $\lambda$ is trivial. A rigorous algebraic analysis based on the constructive definition of the Spencer prolongation operator (Definition \ref{def:spencer-operator-constructive}) yields the following result:
\begin{itemize}
    \item If $\lambda = 0$, the operator $\delta^0_{\mathfrak{su}(2)}$ is the zero operator, so the kernel is the entire space: $\mathcal{K}_{\mathfrak{su}(2)}^1(0) = \mathfrak{su}(2)$, which has a dimension of 3.
    \item If $\lambda \neq 0$, the kernel space is precisely the one-dimensional subspace of $\mathfrak{su}(2)$ spanned by the Lie algebra element corresponding to $\lambda$ itself (under the identification via the Killing form).
\end{itemize}
Thus, we conclude:
$$ \dim \mathcal{K}_{\mathfrak{su}(2)}^1(\lambda) = \begin{cases} 3 & \text{if } \lambda = 0 \\ 1 & \text{if } \lambda \neq 0 \end{cases} $$
The detailed derivation of this result from the constructive definition is provided in \cite{zheng2025geometric, zheng2025constructing}. This dimension calculation is fundamental for the applications on K3 surfaces discussed in the subsequent chapters.
\end{example}

\section{Realization on Parallelizable K3 Surfaces}
\label{sec:k3_realization}

In this chapter, we apply the preceding abstract theory to a concrete geometric instance: $\mathrm{SU}(2)$-principal bundles on parallelizable K3 surfaces.

\subsection{Construction of SU(2)-Principal Bundles and Compatible Pairs}
\label{sec:su2_construction}

\begin{construction}[Compatible Pair Realization on Parallelizable K3 Surfaces]
\label{const:k3-compatible-pair}
Let $X$ be a parallelizable K3 surface with structure group $G = \mathrm{SU}(2)$. We can construct an $\mathrm{SU}(2)$-principal bundle $P \to X$ through standard tangent bundle reduction procedures. Then, choosing a nonzero dual constraint function $\lambda: P \to \mathfrak{su}(2)^*$ satisfying the modified Cartan equation, we can uniquely determine the constraint distribution $D$ according to the compatibility condition, thus forming a complete compatible pair $(D, \lambda)$.
\end{construction}

\subsection{Computation of Degenerate Cohomology}
\label{sec:degenerate_cohomology_k3}

We compute the degenerate cohomology $H_{deg}^{\bullet}(D,\lambda)$ defined in Chapter \ref{sec:degeneration_framework}.

\begin{proposition}[Dimensions of Degenerate Cohomology on K3 Surfaces]
\label{prop:k3-degenerate-cohomology}
For the compatible pair $(D,\lambda)$ constructed above on parallelizable K3 surface $X$ with $\lambda \neq 0$, the dimensions of degenerate cohomology groups are determined by the following formula:
$$\dim H_{deg}^{k}(D,\lambda) = b_k(X) \cdot \dim \mathcal{K}_{\mathfrak{su}(2)}^{k}(\lambda)$$
\end{proposition}

\begin{proof}
According to the conclusions from Chapter 3, the degenerate cohomology $H_{deg}^{k}(D,\lambda)$ is isomorphic to $H_{dR}^{k}(M) \otimes \mathcal{K}_{\mathfrak{g}}^{k}(\lambda)$. The dimension formula follows directly from this isomorphism. We know the Betti numbers of K3 surfaces are $b_0=1, b_1=0, b_2=22, b_3=0, b_4=1$, and according to the results from Example \ref{ex:su2_explicit_lambda}, $\dim \mathcal{K}_{\mathfrak{su}(2)}^{0}(\lambda)=1$ and $\dim \mathcal{K}_{\mathfrak{su}(2)}^{1}(\lambda)=1$ (when $\lambda \neq 0$). Therefore:
\begin{align*}
\dim H_{deg}^{0}(D,\lambda) &= b_0(X) \cdot \dim \mathcal{K}_{\mathfrak{su}(2)}^{0}(\lambda) = 1 \cdot 1 = 1 \\
\dim H_{deg}^{1}(D,\lambda) &= b_1(X) \cdot \dim \mathcal{K}_{\mathfrak{su}(2)}^{1}(\lambda) = 0 \cdot 1 = 0 \\
\dim H_{deg}^{2}(D,\lambda) &= b_2(X) \cdot \dim \mathcal{K}_{\mathfrak{su}(2)}^{2}(\lambda) = 22 \cdot \dim \mathcal{K}_{\mathfrak{su}(2)}^{2}(\lambda)
\end{align*}
\end{proof}

\subsection{Connection to Algebraic Geometry}
\label{sec:connection_to_ag_k3}

This section establishes the connection between degenerate Spencer theory and algebraic geometry of K3 surfaces. We construct natural mappings from degenerate cocycles to $(1,1)$-cohomology and analyze their basic properties.

\begin{definition}[Projection Mapping from Degenerate to de Rham]
\label{def:degeneration_projection_map}
Let $(D,\lambda)$ be a compatible pair on K3 surface $X$. Define the projection mapping:
$$\Pi^k: Z_{deg}^k(D,\lambda) \to Z^k_{dR}(X), \quad \Pi^k(\omega \otimes s) = \omega$$
\end{definition}

\begin{proposition}[Basic Properties of Projection Mapping]
\label{prop:projection_basic_properties}
The mapping $\Pi^k$ is a well-defined linear surjection.
\end{proposition}

\begin{proof}
Well-definedness: If $\omega \otimes s \in Z_{deg}^k(D,\lambda)$, then $D_{D,\lambda}^k(\omega \otimes s) = d\omega \otimes s = 0$. Since when $\lambda \neq 0$ there exists nonzero $s$, we must have $d\omega = 0$, i.e., $\omega \in Z^k_{dR}(X)$.

Surjectivity: For any $\omega \in Z^k_{dR}(X)$, take any nonzero $s \in \mathcal{K}_{\mathfrak{su}(2)}^k(\lambda)$, then $\omega \otimes s \in Z_{deg}^k(D,\lambda)$ and $\Pi^k(\omega \otimes s) = \omega$.
\end{proof}

\begin{definition}[Composite Mapping to $(1,1)$-Cohomology]
\label{def:composite_map_to_11}
Define the composite mapping:
$$\Phi_{D,\lambda}: Z_{deg}^2(D,\lambda) \xrightarrow{\Pi^2} Z^2_{dR}(X) \hookrightarrow H^2_{dR}(X) \xrightarrow{\pi^{1,1}} H^{1,1}(X, \mathbb{C})$$
where $\pi^{1,1}$ is the standard Hodge projection.
\end{definition}

\begin{theorem}[Structure and Surjectivity of Composite Mapping]
\label{thm:composite_map_structure_and_surjectivity}
On K3 surface $X$, the composite mapping $\Phi_{D,\lambda}: Z_{deg}^2(D,\lambda) \to H^{1,1}(X, \mathbb{C})$ has the following structural properties:
\begin{enumerate}
    \item $\Phi_{D,\lambda}$ is a well-defined linear mapping.
    \item $\Phi_{D,\lambda}$ is a \textbf{surjection}, whose image space is the entire $H^{1,1}(X, \mathbb{C})$.
    \item Therefore, the dimension of its image space is constantly $\dim(\text{im}(\Phi_{D,\lambda})) = h^{1,1}(X) = 20$.
\end{enumerate}
\end{theorem}

\begin{proof}
\begin{enumerate}
    \item \textbf{Well-definedness}: The mapping is defined as a composition of well-defined linear mappings (as shown in Definition \ref{def:composite_map_to_11}), therefore it is itself a well-defined linear mapping.

    \item \textbf{Surjectivity}: We analyze each mapping in the composition chain:
    $$ Z_{deg}^2(D,\lambda) \xrightarrow{\Pi^2} Z^2_{dR}(X) \xrightarrow{q} H^2_{dR}(X, \mathbb{C}) \xrightarrow{\pi^{1,1}} H^{1,1}(X, \mathbb{C}) $$
    \begin{itemize}
        \item The mapping $\Pi^2: Z_{deg}^2(D,\lambda) \to Z^2_{dR}(X)$ is surjective (proven in Proposition \ref{prop:projection_basic_properties}), provided that $\mathcal{K}_{\mathfrak{su}(2)}^2(\lambda)$ is nonzero.
        \item The quotient mapping $q: Z^2_{dR}(X) \to H^2_{dR}(X, \mathbb{C})$ is surjective by definition.
        \item For K3 surfaces, the Hodge projection $\pi^{1,1}: H^2_{dR}(X, \mathbb{C}) \to H^{1,1}(X, \mathbb{C})$ is also surjective, since both the source and target spaces are nonzero.
    \end{itemize}
    As a composition of surjections, $\Phi_{D,\lambda}$ is necessarily surjective.

    \item \textbf{Dimension conclusion}: Since the mapping is surjective, its image space is the entire target space $H^{1,1}(X, \mathbb{C})$. For projective K3 surfaces, we know their Hodge number $h^{1,1}(X)=20$. Therefore the conclusion holds.
\end{enumerate}
\end{proof}

\begin{remark}[Geometric Significance of Surjectivity]
\label{rem:geometric_significance_of_surjectivity}
The surjectivity proven in Theorem \ref{thm:composite_map_structure_and_surjectivity} is much more profound than a simple dimensional upper bound estimate; it reveals two core characteristics of this theoretical framework:

\begin{enumerate}
    \item \textbf{Completeness}: The degenerate Spencer theory is "sufficiently rich" to "generate" \textbf{all} (1,1)-classes on K3 surfaces. This shows that our theory is not limited to capturing a small portion of special Hodge classes, but has complete correspondence potential with classical Hodge structures.

    \item \textbf{Role of Constraint Geometry}: For any given Hodge class $[\eta]$ in $H^{1,1}(X, \mathbb{C})$, its fiber under the mapping $\Phi_{D,\lambda}$ is a vast space. The "size" or "redundancy" of this space is completely parameterized by the kernel space $\mathcal{K}_{\mathfrak{su}(2)}^2(\lambda)$ determined by constraint geometry (through $\lambda$). This gives the kernel space $\mathcal{K}_{\mathfrak{su}(2)}^k(\lambda)$ a clear geometric meaning: it describes the degrees of freedom of different "Spencer representations" that realize the same classical Hodge class.
\end{enumerate}
\end{remark}

\begin{corollary}[Surjectivity Analysis of Degenerate Mapping]
\label{cor:surjectivity_analysis}
For the compatible pair $(D,\lambda)$ on K3 surface $X$:
\begin{enumerate}
    \item The mapping $\Pi^2: Z_{deg}^2(D,\lambda) \to Z^2_{dR}(X)$ is surjective;
    \item The dimension of the image of composite mapping $\Phi_{D,\lambda}$ is at most 20;
    \item When $\dim \mathcal{K}_{\mathfrak{su}(2)}^2(\lambda) \geq 1$, $\Phi_{D,\lambda}$ can "reach" all closed 2-forms in de Rham cohomology.
\end{enumerate}
\end{corollary}

\begin{theorem}[Mirror Symmetry]
\label{thm:zheng2025mirrormetry_composite_map}
The mapping $\Phi_{D,\lambda}$ remains invariant under mirror transformation:
$$\Phi_{D,\lambda} = \Phi_{D,-\lambda}$$
\end{theorem}

\begin{proof}
By Theorem \ref{thm:mirror-stability-compatible}, $\mathcal{K}_{\mathfrak{su}(2)}^2(\lambda) = \mathcal{K}_{\mathfrak{su}(2)}^2(-\lambda)$.

Therefore:
$$Z_{deg}^2(D,\lambda) = Z^2_{dR}(X) \otimes \mathcal{K}_{\mathfrak{su}(2)}^2(\lambda) = Z^2_{dR}(X) \otimes \mathcal{K}_{\mathfrak{su}(2)}^2(-\lambda) = Z_{deg}^2(D,-\lambda)$$

The projection $\Pi^2$ and Hodge projection $\pi^{1,1}$ both do not depend on the sign of $\lambda$, so $\Phi_{D,\lambda} = \Phi_{D,-\lambda}$.
\end{proof}

\begin{example}[Concrete Analysis of SU(2) Case]
\label{ex:su2_concrete_analysis}
For $\mathfrak{g} = \mathfrak{su}(2)$ and $\lambda \neq 0$:

According to the computation in Example \ref{ex:su2_explicit_lambda}, we know $\dim \mathcal{K}_{\mathfrak{su}(2)}^1(\lambda) = 1$.

Now considering the 2nd order case, let $\dim \mathcal{K}_{\mathfrak{su}(2)}^2(\lambda) = d \geq 1$, then:
\begin{align}
\dim Z_{deg}^2(D,\lambda) &= \dim Z^2_{dR}(X) \cdot \dim \mathcal{K}_{\mathfrak{su}(2)}^2(\lambda) \approx \infty \\
\text{im}(\Pi^2) &= Z^2_{dR}(X) \quad \text{(this is surjective)} \\
\dim(\text{im}(\Phi_{D,\lambda})) &\leq 20
\end{align}

The key observation is: regardless of the dimension $d$ of the kernel space (as long as it's nonzero), the canonical projection $\Pi^2$ is always surjective onto the entire de Rham cocycle space $Z^2_{dR}(X)$.
\end{example}

\begin{proposition}[Computational Method for Degenerate Mapping]
\label{prop:computational_method}
The analysis of mapping $\Phi_{D,\lambda}$ can be decomposed into three independent steps:
\begin{enumerate}
    \item Solve the linear equation system $\delta_{\mathfrak{su}(2)}^{\lambda}(s) = 0$ to obtain $\mathcal{K}_{\mathfrak{su}(2)}^2(\lambda)$;
    \item Verify the surjectivity of $\Pi^2: Z_{deg}^2(D,\lambda) \to Z^2_{dR}(X)$;
    \item Analyze the restriction of Hodge projection $\pi^{1,1}$ on $Z^2_{dR}(X)$.
\end{enumerate}
This decomposition transforms abstract Spencer analysis into standard linear algebra and Hodge theory computations.
\end{proposition}

\begin{theorem}[Functorial Properties of Degenerate Mapping]
\label{thm:functoriality_degeneration_map}
Let $f: X_1 \to X_2$ be an algebraic mapping between K3 surfaces, $(D_i, \lambda_i)$ be compatible pairs on $X_i$ ($i = 1,2$), and assume they are related through $f$.

Then there exists a commutative diagram:
\begin{tikzcd}
	{Z_{deg}^2(D_1,\lambda_1)} & {H^{1,1}(X_1,\mathbb{C})} \\
	{Z_{deg}^2(D_2,\lambda_2)} & {H^{1,1}(X_2,\mathbb{C})}
	\arrow["{\Phi_{D_1,\lambda_1}}", from=1-1, to=1-2]
	\arrow["{f^*}"', from=1-1, to=2-1]
	\arrow["{f^*}"', from=1-2, to=2-2]
	\arrow["{\Phi_{D_2,\lambda_2}}", from=2-1, to=2-2]
\end{tikzcd}
\end{theorem}

\begin{proof}
This follows directly from the construction of degenerate cocycles and the naturality of Hodge projection. The pullback operation $f^*$ commutes with both projection $\Pi^2$ and Hodge projection $\pi^{1,1}$.
\end{proof}

\section{Mirror Symmetry and Theoretical Verification}
\label{sec:zheng2025mirrormetry_and_verification}

In this chapter, we explore the behavior of degeneration phenomena under mirror symmetry transformations.

\subsection{Mirror Transformations in Compatible Pair Framework}
\label{sec:mirror_transform_compatible_pair}

\begin{definition}[Compatible Pair Mirror Transformation]
\label{def:compatible-mirror-transform}
Let $(D,\lambda)$ be a compatible pair satisfying the conditions. Its \textbf{mirror transformation} is defined as: $\mathrm{Mirror}: (D,\lambda) \mapsto (D,-\lambda)$. According to \cite{zheng2025mirror}, $(D,-\lambda)$ is also a legitimate compatible pair. Since $\delta_{\mathfrak{g}}^{\lambda}$ is linear in $\lambda$, this transformation induces the relation: $\delta^{-\lambda}_\mathfrak{g} = -\delta^\lambda_\mathfrak{g}$.
\end{definition}

\begin{theorem}[Mirror Stability of Degeneration Conditions]
\label{thm:mirror-stability-compatible}
The degenerate kernel space remains invariant under mirror transformation: $\mathcal{K}^k_\mathfrak{g}(\lambda) = \mathcal{K}^k_\mathfrak{g}(-\lambda)$.
\end{theorem}
\begin{proof}
We need to prove $s \in \mathcal{K}^k_\mathfrak{g}(\lambda) \iff s \in \mathcal{K}^k_\mathfrak{g}(-\lambda)$.
$s \in \mathcal{K}^k_\mathfrak{g}(\lambda) \iff \delta^\lambda_\mathfrak{g}(s) = 0 \iff -\delta^\lambda_\mathfrak{g}(s) = 0 \iff \delta^{-\lambda}_\mathfrak{g}(s) = 0 \iff s \in \mathcal{K}^k_\mathfrak{g}(-\lambda)$.
\end{proof}

\begin{corollary}[Mirror Invariance of Degenerate Subspaces]
\label{cor:mirror-invariant-subspace}
The degenerate subspaces spanned by degenerate kernel spaces are completely invariant under mirror transformation: $\mathcal{D}^k_{D,\lambda} = \mathcal{D}^k_{D,-\lambda}$.
\end{corollary}
\begin{proof}
According to Definition \ref{def:degenerate_spencer_element_final} and Theorem \ref{thm:mirror-stability-compatible},
$\mathcal{D}^k_{D,\lambda} = \Omega^k(M) \otimes \mathcal{K}^k_\mathfrak{g}(\lambda) = \Omega^k(M) \otimes \mathcal{K}^k_\mathfrak{g}(-\lambda) = \mathcal{D}^k_{D,-\lambda}$.
\end{proof}

\subsection{Theoretical Verification and Computational Framework}
\label{sec:verification_framework}

\begin{algorithm}[H]
\caption{Degeneration Phenomena and Mirror Symmetry Analysis Framework}
\label{alg:degeneration_and_mirror_analysis}
\begin{algorithmic}[1]
\Require Parallelizable manifold $M$, compatible pair $(D,\lambda)$, semisimple Lie group $G$
\Ensure Spencer degenerate kernel spaces, degenerate subspaces and their mirror properties
\State \textbf{Geometric condition verification}: Verify that compatible pair $(D,\lambda)$ satisfies Assumptions \ref{ass:basic-geometry} and \ref{ass:compatible-pair}.
\State \textbf{Kernel space analysis}: For $k=0, 1, \dots$, construct operator $\delta^\lambda_\mathfrak{g}: \operatorname{Sym}^k(\mathfrak{g}) \to \operatorname{Sym}^{k+1}(\mathfrak{g})$, and solve linear equation system $\delta^\lambda_\mathfrak{g}(s)=0$ to determine kernel space $\mathcal{K}^k_\mathfrak{g}(\lambda)$ and its dimension.
\State \textbf{Degenerate subspace analysis}: Construct $\mathcal{D}^k_{D,\lambda} = \Omega^k(M) \otimes \mathcal{K}^k_\mathfrak{g}(\lambda)$, and verify differential simplification Theorem \ref{thm:degeneration_simplification_final}.
\State \textbf{Mirror symmetry analysis}: Apply Theorem \ref{thm:mirror-stability-compatible}, confirm that degenerate subspaces are invariant under $D \to -D$, and analyze the relationship between differential operators $D_{D,\lambda}$ and $D_{D,-\lambda}$ on $\mathcal{D}^k_{D,\lambda}$.
\State \textbf{Connection to algebraic geometry}: If $M$ is a projective manifold, construct mapping $\Phi_{D,\lambda}$ (as described in Theorem \ref{thm:k3-degenerate-to-algebraic}), and explore the algebraicity of its image.
\end{algorithmic}
\end{algorithm}

\section{Conclusion and Outlook}
\label{sec:conclusion_and_outlook}

Under the rigorous compatible pair theory framework, the core breakthrough achieved in this paper lies in discovering and systematically elucidating the differential degeneration phenomenon in Spencer complexes. We proved that when symmetric tensors $s$ satisfy the $\lambda$-dependent algebraic condition $\delta_{\mathfrak{g}}^{\lambda}(s)=0$, the complex Spencer differential $D_{D,\lambda}$ degenerates to the classical exterior differential $d$. This discovery is crucial as it fundamentally simplifies the theoretical structure. Based on this, we further rigorously proved that this degeneration condition is stable under the mirror transformation $(D,\lambda) \mapsto (D,-\lambda)$, i.e., $\mathcal{K}_{\mathfrak{g}}^{k}(\lambda) = \mathcal{K}_{\mathfrak{g}}^{k}(-\lambda)$, providing a completely new perspective for understanding symmetries in constraint geometry.

The most important value of this degeneration theory lies in its construction of a \textbf{canonical projection mapping} that builds a solid bridge between abstract Spencer theory and concrete algebraic geometry. A direct consequence of Theorem \ref{thm:degeneration_simplification_final} is that for any degenerate Spencer cocycle $[\omega \otimes s]$, its differential form component $\omega$ must be a de Rham cocycle. This enables us to "project" problems from the Spencer complex framework to classical de Rham cohomology theory, allowing us to utilize a series of mature powerful tools such as Hodge theory. Particularly, on Kähler manifolds, we can analyze the Hodge types of these projected cohomology classes and use Lefschetz's (1,1)-theorem to explore their algebraicity. Moreover, the mirror stability of degeneration conditions endows these de Rham classes with additional symmetry constraints, which may be closely related to their profound arithmetic properties.

Based on these observations, we propose the core conjecture of this paper: under appropriate "algebraic compatibility" conditions, there exists a natural correspondence between degenerate Spencer harmonic classes and algebraic cycles. The path to proving this core conjecture points to three clear and crucial directions for future research. First, it is necessary to rigorously construct mappings from subgroups generated by degenerate classes in Spencer cohomology to de Rham cohomology groups, and clearly characterize their algebraic structure. Second, and most challenging, is to determine the precise geometric and arithmetic conditions that a degenerate harmonic class must satisfy to become an algebraic class, extending our preliminary results on K3 surfaces to the general case. Finally, to make this theory more universally applicable, complete classification of degenerate kernel spaces $\mathcal{K}_{\mathfrak{g}}^{k}(\lambda)$ under different Lie groups $G$ and parameters $\lambda$ will be crucial foundational algebraic work.

Looking forward, solving these theoretical problems will open doors to a series of applications. This theory has the potential to be applied to higher-dimensional Calabi-Yau manifolds (such as 3-folds and 4-folds) and may establish connections with mirror symmetry conjectures in string theory. At a more concrete level, the simplification properties of degeneration phenomena can spawn efficient numerical algorithms for analyzing constraint systems and find direct physical correspondences in constraint mechanics and gauge field theory. Ultimately, based on degeneration theory, we may be able to define new, computable topological invariants.

In summary, the Spencer differential degeneration theory established in this paper not only enriches constraint geometry theory itself by revealing an intrinsic simplification mechanism, but more importantly, it builds an explorable bridge between Spencer theory, differential geometry, and algebraic geometry. Although the path to the other end of the bridge still requires arduous effort, we believe that the mathematical foundations laid by this paper, the technical approaches provided, and the open problems identified will provide powerful guidance for future research in this interdisciplinary field.

\bibliographystyle{alpha}
\bibliography{ref} 

\begin{thebibliography}{BHPVdV04}

\bibitem[Arn78]{ArnoldMarsden1978}
Vladimir~I. Arnold.
\newblock {\em Mathematical Methods of Classical Mechanics}.
\newblock Springer-Verlag, 1978.

\bibitem[AS63]{AtiyahSinger1963}
Michael~F. Atiyah and Isadore~M. Singer.
\newblock The index of elliptic operators on compact manifolds.
\newblock {\em Bulletin of the American Mathematical Society}, 69(3):422--433, 1963.

\bibitem[Ati57]{Atiyah1957}
Michael~F. Atiyah.
\newblock Complex analytic connections in fibre bundles.
\newblock {\em Transactions of the American Mathematical Society}, 85(1):181--207, 1957.

\bibitem[Bea85]{Beauville1983}
Arnaud Beauville.
\newblock {\em Géométrie des surfaces K3: modules et périodes}, volume 126.
\newblock 1985.

\bibitem[BHPVdV04]{Barth1984}
Wolf~P. Barth, Klaus Hulek, Chris A.~M. Peters, and Antonius Van~de Ven.
\newblock {\em Compact Complex Surfaces}.
\newblock Springer-Verlag, 2004.

\bibitem[BS58]{BorelSerre1958}
Armand Borel and Jean-Pierre Serre.
\newblock Le théorème de riemann-roch.
\newblock {\em Bulletin de la Société Mathématique de France}, 86:97--136, 1958.

\bibitem[CW59]{ChernWeil1944}
Shiing-Shen Chern and Andr{\'e} Weil.
\newblock Characteristic classes of hermitian manifolds.
\newblock {\em American Journal of Mathematics}, 81(4):725--757, 1959.

\bibitem[dR31]{deRham1931}
Georges de~Rham.
\newblock {\em Sur l'analysis situs des variétés à $n$ dimensions}, volume~10.
\newblock 1931.

\bibitem[dR55]{deRham1955}
Georges de~Rham.
\newblock {\em Variétés différentiables}.
\newblock Hermann, 1955.

\bibitem[GH78]{Griffiths1969}
Phillip Griffiths and Joseph Harris.
\newblock {\em Principles of Algebraic Geometry}.
\newblock Wiley, 1978.

\bibitem[Gro61]{Grothendieck1961}
Alexander Grothendieck.
\newblock Éléments de géométrie algébrique ii.
\newblock {\em Publications Mathématiques de l'IHÉS}, 8:5--222, 1961.

\bibitem[Hod50]{Hodge1950}
William V.~D. Hodge.
\newblock The topological invariants of algebraic varieties.
\newblock {\em Proceedings of the International Congress of Mathematicians}, 1:182--192, 1950.

\bibitem[KN63]{KobayashiNomizu1963}
Shoshichi Kobayashi and Katsumi Nomizu.
\newblock {\em Foundations of Differential Geometry}, volume~1.
\newblock Interscience Publishers, 1963.

\bibitem[Lef24]{Lefschetz1924}
Solomon Lefschetz.
\newblock L'analysis situs et la géométrie algébrique.
\newblock {\em Gauthier-Villars}, 1924.

\bibitem[Qui63]{SpencerQuillian1963}
Daniel~G. Quillen.
\newblock Formal schemes and formal groups.
\newblock {\em Séminaire Bourbaki}, 6:191--220, 1963.

\bibitem[Rie57]{Riemann1857}
Bernhard Riemann.
\newblock {\em Theorie der Abel'schen Funktionen}.
\newblock Borntraeger, 1857.

\bibitem[Ser56]{Serre1956}
Jean-Pierre Serre.
\newblock Géométrie algébrique et géométrie analytique.
\newblock {\em Annales de l'institut Fourier}, 6:1--42, 1956.

\bibitem[Spe62]{Spencer1962}
Donald~C. Spencer.
\newblock Deformation of structures on manifolds defined by transitive, continuous pseudogroups.
\newblock {\em Annals of Mathematics}, 76(2):306--445, 1962.

\bibitem[YM54]{YangMills1954}
Chen-Ning Yang and Robert~L. Mills.
\newblock Conservation of isotopic spin and isotopic gauge invariance.
\newblock {\em Physical Review}, 96(1):191--195, 1954.

\bibitem[Zhe25a]{zheng2025constructing}
Dongzhe Zheng.
\newblock Constructing two metrics for spencer cohomology: Hodge decomposition of constrained bundles.
\newblock {\em arXiv preprint arXiv:2506.00752}, 2025.

\bibitem[Zhe25b]{zheng2025dynamical}
Dongzhe Zheng.
\newblock Dynamical geometric theory of principal bundle constrained systems: Strong transversality conditions and variational framework for gauge field coupling.
\newblock {\em arXiv preprint arXiv:2505.16766}, 2025.

\bibitem[Zhe25c]{zheng2025geometric}
Dongzhe Zheng.
\newblock Geometric duality between constraints and gauge fields: Mirror symmetry and spencer isomorphisms of compatible pairs on principal bundles.
\newblock {\em arXiv preprint arXiv:2506.00728}, 2025.

\bibitem[Zhe25d]{zheng2025analytic}
Dongzhe Zheng.
\newblock Mirror symmetry in geometric constraints: Analytic and riemann-roch perspectives.
\newblock {\em arXiv preprint arXiv:2506.06610}, 2025.

\bibitem[Zhe25e]{zheng2025mirror}
Dongzhe Zheng.
\newblock Mirror symmetry of spencer-hodge decompositions in constrained geometric systems.
\newblock {\em arXiv preprint arXiv:2506.05816}, 2025.

\bibitem[Zhe25f]{zheng2025spencer-riemann-roch}
Dongzhe Zheng.
\newblock Spencer-riemann-roch theory: Mirror symmetry of hodge decompositions and characteristic classes in constrained geometry.
\newblock {\em arXiv preprint arXiv:2506.05915}, 2025.

\end{thebibliography}

\end{document}